\documentclass[11pt]{amsart}

\usepackage[T1]{fontenc}
\usepackage{lmodern}
\usepackage{microtype}
\usepackage{amsmath,amssymb,amsthm,mathtools}
\usepackage{booktabs}
\usepackage{array}
\usepackage{enumitem}
\usepackage{xcolor}
\usepackage[colorlinks=true,linkcolor=blue!55!black,citecolor=blue!55!black,urlcolor=blue!55!black]{hyperref}

\allowdisplaybreaks
\setlist[itemize]{leftmargin=2em,itemsep=2pt,topsep=4pt}
\setlist[enumerate]{leftmargin=2.2em,itemsep=2pt,topsep=4pt}

\newtheorem{theorem}{Theorem}[section]
\newtheorem{proposition}[theorem]{Proposition}
\newtheorem{lemma}[theorem]{Lemma}
\newtheorem{corollary}[theorem]{Corollary}
\newtheorem{criterion}[theorem]{Criterion}
\theoremstyle{definition}
\newtheorem{definition}[theorem]{Definition}
\newtheorem{question}[theorem]{Question}

\theoremstyle{remark}
\newtheorem{remark}[theorem]{Remark}

\DeclareMathOperator{\depth}{depth}
\DeclareMathOperator{\Ass}{Ass}
\DeclareMathOperator{\ann}{ann}
\DeclareMathOperator{\rank}{rank}
\DeclareMathOperator{\Span}{span}
\DeclareMathOperator{\Soc}{Soc}

\DeclareMathOperator{\Hom}{Hom}
\DeclareMathOperator{\Ext}{Ext}
\DeclareMathOperator{\res}{res}

\DeclareMathOperator{\Supp}{Supp}

\newcommand{\F}{\mathbb F}
\newcommand{\kbar}{\overline{\mathbb F}_{2}}

\newcommand{\Z}{\mathbb Z}
\newcommand{\SG}[2]{\operatorname{SmallGroup}(#1,#2)}
\newcommand{\EA}{\mathcal E}
\newcommand{\Nrad}{\sqrt{0}}

\newcommand{\Res}{\operatorname{Res}}

\title[Centralizer excess and Carlson's conjecture]{Centralizer Excess as an Obstruction to Carlson's Depth Conjecture}
\author{Xinan Dai}
\address{
Key Laboratory for Information Science of Electromagnetic Waves,
College of Future Information and Technology,
Fudan University, Shanghai, China
}
\email{xndai23@m.fudan.edu.cn}
\author{Kuok Fai Chao}
\address{Lui Che Woo College, University of Macau, Macau, China}
\email{kchao@um.edu.mo}
\date{August 5, 2026}


\begin{document}

\begin{abstract}
Let $G=\SG{128}{859}$ and $k=\kbar$.  The cohomology ring $H^*(G;k)$ has depth two, and we prove that the minimum quotient dimension of an associated prime is exactly three.  Okuyama's theorem shows that an integer $r$ occurs as such a dimension exactly when there is an elementary abelian subgroup $E\leq G$ of rank $r$ with
\[
\depth H^*(C_G(E);k)=r.
\]
We use this equivalence to define the centralizer excess.  If $d=\depth H^*(K;k)$, Carlson's equality holds precisely when some rank-$d$ subgroup has zero excess.  For $G$, all rank-two centralizers have positive excess.  A complete enumeration of the thirty-one actual rank-three elementary abelian subgroups finds six zero-excess witnesses.  Hence
\[
\omega_a\bigl(H^*(G;k)\bigr)=3.
\]
Since
\[
H^*(G\times(C_2)^n;k)\cong H^*(G;k)[u_1,\dots,u_n],
\]
the standard behavior of associated primes under polynomial extension gives
\[
\omega_a\bigl(H^*(G\times(C_2)^n;k)\bigr)=n+3
\qquad(n\geq0).
\]

We also study the class $\alpha_0=g+fc\in H^3(G;\F_2)$.  It is killed by two degree-one classes but restricts nontrivially to a rank-four elementary abelian subgroup.  It follows that
\[
\dim H^*(G;\F_2)/\ann(\alpha_0)=4.
\]
Thus two explicit linear annihilators do not force a two-dimensional cyclic support.  The assertion is about Krull dimension; it does not say that the support is the whole spectrum.
\end{abstract}

\maketitle

\section{Introduction}
\label{sec:introduction}

Let $K$ be a finite group and let $k$ be a field whose characteristic $p$ divides $|K|$.  The graded algebra
\[
H^*(K;k)=\Ext^*_{kK}(k,k)
\]
is a finitely generated graded $k$-algebra by Evens' theorem, and hence is Noetherian.  We shall use its depth, associated primes, local cohomology, and support \cite{Evens1961,EvensBook}.

Two classical theorems connect these commutative-algebra invariants with subgroup geometry.  Quillen proved that the spectrum of $H^*(K;k)$ is controlled, up to nilpotents, by restriction to elementary abelian $p$-subgroups; in particular,
\[
\dim H^*(K;k)=\rank_p(K)
\]
\cite{QuillenI,QuillenII,QuillenVenkov}.  Duflot proved that central $p$-subgroups provide regular elements and obtained
\[
\depth H^*(K;k)\geq \rank_p Z(S),
\]
where $S$ is a Sylow $p$-subgroup of $K$ \cite{Duflot1981}.  Thus maximal elementary abelian rank controls Krull dimension, whereas central elementary abelian rank supplies a lower bound for depth.

Carlson asked whether depth is always realized by an associated prime \cite[Question~3.1]{Carlson1995}.  Put
\[
A_K=H^*(K;k),
\qquad
\omega_a(A_K)=\min_{\mathfrak p\in\Ass A_K}\dim A_K/\mathfrak p.
\]
The standard inequality
\[
\depth A_K\leq\omega_a(A_K)
\]
shows that Carlson's question is an attainment problem.

\begin{question}[Carlson]
\label{q:carlson}
Does every finite group in modular characteristic satisfy
\[
\depth H^*(K;k)=\omega_a\bigl(H^*(K;k)\bigr)?
\]
\end{question}

Carlson's equality is known for several families of groups \cite{Green2003,Schafer2019,Garaialde2022}; see also \cite{Schafer2025} for a recent formulation in terms of centralizers.  Green and King computed the mod-$2$ cohomology rings of all groups of order $128$.  Their article and online database provide the presentations, Poincar\'e series, depth data, and restriction maps used below \cite{GreenKing2011,GreenKingDatabase}.

Using this infrastructure, one of the present authors, Dai, together with Deng, Shi, Wu, and Yang~\cite{DaiEtAl2026}, gave a negative answer to Question~\ref{q:carlson} for
\[
G=\SG{128}{859}.
\]
Put
\[
A_0:=H^*(G;\F_2),
\qquad
A:=\kbar\otimes_{\F_2}A_0
   \cong H^*(G;\kbar).
\]
All explicit presentations, restriction maps, and finite-dimensional calculations are carried out in $A_0$; the scalar extension $A$ is used when applying results stated over an algebraically closed field.  They proved
\[
\depth A=2,
\qquad
\omega_a(A)\geq3.
\]
Their proof has two steps.  First, a regular sequence followed by a nonzero socle class certifies that the ambient depth is two.  Second, Okuyama's theorem converts any hypothetical associated prime of quotient dimension two into a rank-two elementary abelian subgroup $E$ satisfying
\[
\depth H^*(C_G(E);\kbar)=2.
\]
An exhaustive enumeration gives seventy-five rank-two elementary abelian subgroups and six centralizer types, but every corresponding centralizer has depth at least three.  Hence no associated prime can have quotient dimension two.

A complete enumeration of the thirty-one actual rank-three elementary abelian subgroups yields four centralizer types and six subgroups of zero excess.  Okuyama's converse therefore produces an associated prime of quotient dimension three, and the preceding lower bound gives
\[
\omega_a(A)=3.
\]
For the centralizer type $C_4\times C_2\times C_2$, the depth-three assertion follows directly from Duflot's lower bound and Quillen's dimension theorem.  In the semidirect-product coordinates below, one such witness is
\[
E_3=\langle u^2,v^2,us^2\rangle,
\qquad
C_G(E_3)\cong C_4\times C_2\times C_2.
\]
The complete calculation and this direct check are recorded in Section~\ref{sec:stabilization}.

Okuyama's theorem gives a convenient reformulation of the counterexample.  Define
\[
\varepsilon_K(E;k)=\depth H^*(C_K(E);k)-\rank_p(E)
\]
for an elementary abelian $p$-subgroup $E\leq K$.  If $d=\depth H^*(K;k)$, then Carlson's equality holds if and only if some rank-$d$ subgroup has zero excess.  For the order-$128$ example, every rank-two subgroup has positive excess.  This is the local obstruction used throughout the paper; it is closely related to Schafer's centralizer-depth formulation \cite{Schafer2025}.

The excess behaves well under elementary abelian direct factors.  If it is positive at depth $d$ for $K$, then it is positive at depth $d+n$ for $K\times(C_2)^n$, so
\[
G_n=G\times(C_2)^n
\]
is a counterexample for every $n\geq0$.  Moreover, the K\"unneth isomorphism identifies its cohomology with the polynomial extension $A[u_1,\dots,u_n]$.  Associated primes extend under this faithfully flat polynomial extension, and quotient dimensions increase by $n$.  Hence
\[
\omega_a\bigl(H^*(G_n;\kbar)\bigr)=\omega_a(A)+n=n+3.
\]

We also compare two different kinds of witnesses.  The proof that $\depth A=2$ uses a socle element after quotienting by a regular sequence.  An associated prime, on the other hand, is the annihilator of an element in the original ring.  The class
\[
\alpha_0=g+fc\in H^3(G;\F_2)
\]
shows why these statements should not be conflated.  It satisfies
\[
a\alpha_0=b\alpha_0=0,
\]
but its restriction to a rank-four elementary abelian subgroup is nonzero.  Since the target cohomology ring is a domain,
\[
\dim H^*(G;\F_2)/\ann(\alpha_0)=4.
\]
We also give explicit representatives for the three conjugacy classes of maximal elementary abelian subgroups and compute the restriction ranks and nilpotent kernels in degrees two through four.

The main statements are the following.

\begin{theorem}[Centralizer criterion]
\label{thm:intro-criterion}
Let $K$ be a finite group, let $k$ be an algebraically closed field of characteristic $p\mid |K|$, and put
\[
d=\depth H^*(K;k).
\]
Then Carlson's equality holds for $K$ if and only if there exists an elementary abelian $p$-subgroup $E\leq K$ of rank $d$ such that
\[
\depth H^*(C_K(E);k)=d.
\]
Equivalently,
\[
\depth H^*(K;k)=\omega_a\bigl(H^*(K;k)\bigr)
\quad\Longleftrightarrow\quad
\varepsilon_d(K;k)=0.
\]
\end{theorem}

\begin{theorem}[Elementary-abelian stabilization]
\label{thm:intro-stabilization}
Let $K$ be a finite group over an algebraically closed field $k$ of characteristic $2$, and put
\[
d=\depth H^*(K;k).
\]
If $\varepsilon_d(K;k)>0$, then for every $n\geq0$,
\[
\depth H^*(K\times(C_2)^n;k)=d+n
\]
and
\[
\varepsilon_{d+n}(K\times(C_2)^n;k)>0.
\]
Hence $K\times(C_2)^n$ also violates Carlson's equality.
\end{theorem}

\begin{theorem}[Structure of the order-$128$ example]
\label{thm:intro-synthesis}
Let
\[
\begin{aligned}
G&=\SG{128}{859},\\
A_0&=H^*(G;\F_2),\\
A&=\kbar\otimes_{\F_2}A_0
  \cong H^*(G;\kbar).
\end{aligned}
\]
Then:
\begin{enumerate}[label=\textup{(\alph*)}]
\item $\depth A=2$, $\omega_a(A)=3$, and $\varepsilon_2(G;\kbar)\geq1$.  Thus the associated-prime defect $\omega_a(A)-\depth A$ is exactly one.
\item The group has a split extension
\[
G\cong(C_4\times C_4)\rtimes_\rho(C_4\times C_2),
\]
with the action $\rho$ given in Theorem~\ref{thm:semidirect}.  Its maximal elementary abelian subgroups form three conjugacy classes of ranks $3,4,4$.  For the explicit rank-three representative $E_3=\langle u^2,v^2,us^2\rangle$ one has
\[
C_G(E_3)\cong C_4\times C_2\times C_2.
\]
\item The low-degree dimensions are
\[
\dim A_0^1=3,
\quad
\dim A_0^2=7,
\quad
\dim A_0^3=12,
\quad
\dim A_0^4=18.
\]
Among these classes, the minimal algebra generators consist of three degree-one classes, three new degree-two classes, one new degree-three class, and no new degree-four class.
\item The class $\alpha_0=g+fc$ is nonzero, differs from $g$ by a decomposable class and therefore represents the same new degree-three direction modulo products, and satisfies $a\alpha_0=b\alpha_0=0$.  Nevertheless,
\[
\dim A_0/\ann_{A_0}(\alpha_0)
=
\dim A/\ann_A(1\otimes\alpha_0)
=4.
\]
\item The joint restriction maps to the three maximal elementary abelian classes have ranks $5,7,12$ in degrees $2,3,4$, respectively, and their kernels are the corresponding homogeneous pieces of the nilradical.
\end{enumerate}
\end{theorem}

\begin{theorem}[Exact stabilized family]
\label{thm:intro-exact-family}
For
\[
G_n=\SG{128}{859}\times(C_2)^n,
\qquad k=\kbar,
\]
one has, for every $n\geq0$,
\[
\depth H^*(G_n;k)=n+2,
\qquad
\omega_a\bigl(H^*(G_n;k)\bigr)=n+3.
\]
Thus every member of the family violates Carlson's equality by exactly one.
\end{theorem}

The depth statement, the lower bound $\omega_a(A)\geq3$, the rank-two enumeration, and the corresponding depth certificates in part~(a) are due to Dai--Deng--Shi--Wu--Yang \cite{DaiEtAl2026}.  We complete the rank-three calculation below, include a direct verification of one depth-three centralizer, and derive the resulting exact formula for the stabilized family.  The remaining assertions are proved here.

Section~\ref{sec:framework} proves the centralizer criterion, and Section~\ref{sec:counterexample} recalls the finite certificates for the counterexample.  Sections~\ref{sec:witnesses}--\ref{sec:detection} treat the group structure, the low-degree ring, and the class $\alpha_0$.  Section~\ref{sec:stabilization} proves the product theorem, records the complete rank-three calculation, derives $\omega_a(A)=3$ and the exact formula $\omega_a(H^*(G_n;k))=n+3$, and gives a direct check of one depth-three centralizer.  The appendix gives the linear-algebra procedure used for the low-degree computations.

\section{Depth, Quillen primes, and centralizer realizability}
\label{sec:framework}

Throughout this section, $K$ is a finite group, $p\mid |K|$, and $k$ is an algebraically closed field of characteristic $p$.  Write
\[
A_K=H^*(K;k),
\qquad
(A_K)_+=\bigoplus_{n>0}H^n(K;k).
\]
At odd primes, prime ideals, associated primes, depth, and support are understood in the standard graded-commutative sense; equivalently, the spectral statements may be read in the even subring.  All explicit computations in the present paper occur at $p=2$, where $A_K$ is an ordinary commutative graded algebra.

\subsection{Depth and associated-prime dimension}

\begin{definition}
A sequence $x_1,\dots,x_d\in(A_K)_+$ of homogeneous elements is $A_K$-regular if multiplication by $x_i$ is injective on
\[
A_K/(x_1,\dots,x_{i-1})
\]
for every $i$, and the final quotient is nonzero.  The maximum possible length is $\depth A_K$.
\end{definition}

\begin{definition}
A prime ideal $\mathfrak p\subseteq A_K$ is associated if $\mathfrak p=\ann_{A_K}(x)$ for some $x\in A_K$.  Define
\[
\omega_a(A_K)=\min_{\mathfrak p\in\Ass A_K}\dim A_K/\mathfrak p.
\]
\end{definition}

For every $\mathfrak p\in\Ass A_K$, localization gives
\[
\depth A_K\leq \depth (A_K)_{\mathfrak p}+\dim A_K/\mathfrak p
=\dim A_K/\mathfrak p;
\]
see, for example, \cite{BrunsHerzog,Eisenbud}.  Consequently,
\begin{equation}
\label{eq:standard-ineq}
\depth A_K\leq\omega_a(A_K).
\end{equation}

\subsection{Quillen primes and Okuyama's equivalence}

Let $E\leq K$ be elementary abelian of rank $r$.  Define the Quillen prime
\[
\mathfrak p_E=(\res^K_E)^{-1}\bigl(\sqrt{0}\,\bigr),
\]
where the nilradical is taken in $H^*(E;k)$.  Quillen's theory gives
\[
\dim A_K/\mathfrak p_E=r
\]
and describes the spectrum of $A_K$ through these elementary abelian strata \cite{QuillenI,QuillenII}.  At $p=2$,
\[
H^*(E;k)\cong k[t_1,\dots,t_r],
\qquad |t_i|=1,
\]
so the target is a domain and its nilradical is zero.

We shall use both directions of the following theorem.

\begin{theorem}[Okuyama \cite{Okuyama2010}]
\label{thm:okuyama}
Let $s\geq1$.
\begin{enumerate}[label=\textup{(\roman*)}]
\item If $\mathfrak p\in\Ass A_K$ and
\[
\dim A_K/\mathfrak p=s,
\]
then there exists an elementary abelian $p$-subgroup $E\leq K$ of rank $s$ such that
\[
\mathfrak p=\mathfrak p_E
\qquad\text{and}\qquad
\depth H^*(C_K(E);k)=s.
\]
\item Conversely, if $E\leq K$ is elementary abelian of rank $s$ and
\[
\depth H^*(C_K(E);k)=s,
\]
then $\mathfrak p_E\in\Ass A_K$.
\end{enumerate}
\end{theorem}

Part~(i) is the implication used in the counterexample of \cite{DaiEtAl2026}.  Part~(ii) shows that the centralizer condition is also sufficient.

\subsection{Realizability spectra and centralizer excess}

\begin{definition}
Define
\begin{align*}
\Sigma_a(K;k)
&=\{\dim A_K/\mathfrak p:\mathfrak p\in\Ass A_K\},\\
\Sigma_c(K;k)
&=\{r:\text{there exists }E\leq K,
\ E\cong(C_p)^r,
\ \depth H^*(C_K(E);k)=r\}.
\end{align*}
\end{definition}

\begin{theorem}[Equality of realizability spectra]
\label{thm:spectra-equality}
For every finite group $K$,
\[
\Sigma_a(K;k)=\Sigma_c(K;k).
\]
\end{theorem}

\begin{proof}
The inclusion $\Sigma_a(K;k)\subseteq\Sigma_c(K;k)$ is Theorem~\ref{thm:okuyama}(i), and the reverse inclusion is Theorem~\ref{thm:okuyama}(ii).
\end{proof}

\begin{definition}
\label{def:excess}
For an elementary abelian $p$-subgroup $E\leq K$, define
\[
\varepsilon_K(E;k)=\depth H^*(C_K(E);k)-\rank_p(E).
\]
For every rank $r$ occurring in $K$, set
\[
\varepsilon_r(K;k)=
\min_{\substack{E\leq K\\E\cong(C_p)^r}}
\varepsilon_K(E;k).
\]
If $d=\depth A_K$, we call $\varepsilon_d(K;k)$ the depth-level centralizer excess.
\end{definition}

Because $E\leq Z(C_K(E))$, choose a Sylow $p$-subgroup of $C_K(E)$ containing $E$.  Duflot's theorem gives
\[
\depth H^*(C_K(E);k)\geq\rank_p(E),
\]
and hence $\varepsilon_K(E;k)\geq0$.

\begin{corollary}[Excess detects associated-prime dimensions]
\label{cor:excess-spectrum}
For every rank $r$ occurring in $K$,
\[
r\in\Sigma_a(K;k)
\quad\Longleftrightarrow\quad
r\in\Sigma_c(K;k)
\quad\Longleftrightarrow\quad
\varepsilon_r(K;k)=0.
\]
Consequently,
\[
\omega_a(A_K)=\min\{r:\varepsilon_r(K;k)=0\}.
\]
\end{corollary}

\begin{proof}
By definition, $\varepsilon_r(K;k)=0$ precisely when some rank-$r$ elementary abelian subgroup has centralizer depth $r$.  The result follows from Theorem~\ref{thm:spectra-equality}.
\end{proof}

\begin{criterion}[Exact centralizer criterion for Carlson's equality]
\label{crit:rank-obstruction}
Let $d=\depth A_K$.  Then
\[
\depth A_K=\omega_a(A_K)
\quad\Longleftrightarrow\quad
\varepsilon_d(K;k)=0.
\]
Equivalently,
\[
\depth A_K<\omega_a(A_K)
\quad\Longleftrightarrow\quad
\varepsilon_d(K;k)>0.
\]
\end{criterion}

\begin{proof}
Inequality~\eqref{eq:standard-ineq} gives $d\leq\omega_a(A_K)$.  By Corollary~\ref{cor:excess-spectrum}, equality holds exactly when $d\in\Sigma_a(K;k)$, which is equivalent to $\varepsilon_d(K;k)=0$.
\end{proof}

Thus failure at depth $d$ means that every rank-$d$ centralizer has depth strictly greater than $d$.  By the converse direction of Okuyama's theorem, no further existence condition is needed: a rank-$r$ subgroup of zero excess already gives an associated Quillen prime of quotient dimension $r$.  Finding a convenient element with that annihilator is a separate, finer question.

\subsection{Low-degree detection}

Let $\EA_{\max}(K)$ be representatives of the conjugacy classes of maximal elementary abelian $p$-subgroups, and define
\[
\Res^n_{\max}:
H^n(K;\F_p)\longrightarrow
\bigoplus_{E\in\EA_{\max}(K)}H^n(E;\F_p).
\]
Every elementary abelian subgroup is contained in a maximal one, and inner automorphisms act trivially on cohomology.  Quillen's $F$-isomorphism theorem therefore gives
\begin{equation}
\label{eq:nil-kernel}
\ker\Res^n_{\max}=\sqrt{0}\cap H^n(K;\F_p).
\end{equation}
Thus $\ker\Res^n_{\max}$ is the degree-$n$ part of the nilradical, and the rank records the classes detected on maximal elementary abelian subgroups.

\section{The exact counterexample}
\label{sec:counterexample}

We recall the counterexample of \cite{DaiEtAl2026}.  The explicit calculations are over \(\F_2\); scalar extension to \(\kbar\) is used for Okuyama's theorem.

\subsection{The group and the main gap}

Retain the notation
\[
\begin{aligned}
G&=\SG{128}{859},\\
A_0&=H^*(G;\F_2),\\
A&=\kbar\otimes_{\F_2}A_0\cong H^*(G;\kbar).
\end{aligned}
\]
The following data are recorded in the Green--King computation and database \cite{GreenKing2011,GreenKingDatabase}.

\begin{table}[ht]
\centering
\caption{Basic structural data for \(G=\SG{128}{859}\).}
\label{tab:basic-data}
\begin{tabular}{@{}ll@{}}
\toprule
Invariant & Value \\
\midrule
Order; exponent & \(128;\ 4\) \\
Minimal number of group generators & \(3\) \\
\(2\)-rank; center \(2\)-rank & \(4;\ 1\) \\
Ranks of maximal elementary abelian classes & \(3,4,4\) \\
Krull dimension; depth; Duflot bound & \(4;\ 2;\ 1\) \\
Minimal ring generators; maximal generator degree & \(13;\ 8\) \\
Minimal relations; maximal relation degree & \(44;\ 14\) \\
\bottomrule
\end{tabular}
\end{table}

The Krull dimension is four because the largest elementary abelian subgroup has rank four.  The depth is two, so the ring has Cohen--Macaulay defect two.  The key question is whether some associated prime realizes quotient dimension two.

Dai--Deng--Shi--Wu--Yang proved the following theorem \cite{DaiEtAl2026}.
\begin{theorem}
\label{thm:counterexample}
For \(G=\SG{128}{859}\),
\[
\depth H^*(G;\kbar)=2,
\qquad
\omega_a\bigl(H^*(G;\kbar)\bigr)\geq3.
\]
In particular, Carlson's associated-prime depth equality fails.
\end{theorem}

\subsection{Complete rank-two enumeration}

A rank-two elementary abelian \(2\)-subgroup is a Klein four group.  It is therefore enough to enumerate all pairs of distinct commuting involutions, retain the generated subgroups of order four, and deduplicate by full element sets.  The procedure is exhaustive because any two distinct nonidentity elements of \(C_2\times C_2\) generate the subgroup.

The enumeration gives thirty-one involutions, seventy-five rank-two elementary abelian subgroups, and twenty-one conjugacy orbits.  Their centralizers have the six types shown in Table~\ref{tab:centralizers}.

\begin{table}[ht]
\centering
\caption{Centralizers of all rank-two elementary abelian subgroups.}
\label{tab:centralizers}
\small
\begin{tabular}{@{}cccc@{}}
\toprule
Centralizer & Multiplicity & \(\rank_2 Z(C_G(E))\) & Certificate \\
\midrule
\(\SG{16}{10}\)  & 12 & 3 & Duflot \\
\(\SG{16}{14}\)  & 32 & 4 & Duflot \\
\(\SG{32}{22}\)  & 24 & 3 & Duflot \\
\(\SG{32}{46}\)  & 4  & 3 & Duflot \\
\(\SG{64}{90}\)  & 2  & 2 & length-3 regular sequence \\
\(\SG{64}{216}\) & 1  & 2 & length-3 regular sequence \\
\midrule
Total & 75 & & \\
\bottomrule
\end{tabular}
\end{table}

Duflot's theorem immediately gives depth at least three for the first four types.  For the two centralizers of order $64$, the center rank is only two, so Duflot's bound does not suffice; their depth is checked by explicit regular sequences.

In the notation of Definition~\ref{def:excess}, the table already yields
\begin{equation}
\label{eq:epsilon2}
\varepsilon_2(G)\geq 1.
\end{equation}
In particular, every rank-two subgroup has positive centralizer excess.

\subsection{The ambient depth-two certificate}

The Green--King presentation has thirteen weighted generators \cite{GreenKingDatabase}.  We use the following aliases:
\[
\begin{array}{c|ccccccccccccc}
\text{alias}&a&b&c&d&e&f&g&h&i&j&k&\ell&m\\
\hline
\text{degree}&1&1&1&2&2&2&3&5&5&6&6&7&8.
\end{array}
\]
Thus \(m=c_{8,65}\) is the Duflot regular generator.  Let \(I\) be the forty-four-relation ideal in the corresponding polynomial ring.  Define
\begin{align}
 f_1&=m,\label{eq:f1}\\
 f_2&=cg+c^4+b^4+fbc+f^2+dc^2+d^2.\label{eq:f2}
\end{align}
The exact ideal-quotient calculations are
\begin{equation}
\label{eq:colon-cert}
(I:f_1)=I,
\qquad
\bigl(I+(f_1):f_2\bigr)=I+(f_1).
\end{equation}
Hence \(f_1,f_2\) is a regular sequence.  In the quotient by this sequence, the class
\begin{equation}
\label{eq:socle-witness}
w=de^2+e^3+ah
\end{equation}
is nonzero and is annihilated by every positive-degree algebra generator.  Therefore
\[
0\neq w\in \Soc\bigl(A_0/(f_1,f_2)\bigr),
\]
so the quotient has depth zero.  It follows that
\[
\depth A_0=2.
\]
Faithfully flat base change gives
\[
H^*(G;\kbar)\cong\kbar\otimes_{\F_2}H^*(G;\F_2)
\quad\text{and}\quad
\depth H^*(G;\kbar)=2.
\]

\subsection{The two exceptional centralizers}

For \(C\cong\SG{64}{90}\), the catalogue presentation admits the regular sequence
\[
c_{1,2},\qquad c_{4,21},\qquad
b_{1,1}^2+b_{2,6}+b_{2,5}+b_{2,4}.
\]
For \(C\cong\SG{64}{216}\), a regular sequence is
\[
c_{2,8},\qquad c_{4,25},\qquad
b_{1,3}^2+b_{1,2}b_{1,3}+b_{1,2}^2+b_{1,0}^2.
\]
In each case, three successive colon-ideal equalities verify the non-zero-divisor conditions.  Thus every centralizer in Table~\ref{tab:centralizers} has depth at least three.

\begin{proof}[Proof of Theorem~\ref{thm:counterexample}]
The ambient certificate gives \(\depth A=2\).  If \(A\) had an associated prime of quotient dimension two, Okuyama's theorem would produce a rank-two elementary abelian subgroup \(E\leq G\) satisfying
\[
\depth H^*(C_G(E);\kbar)=2.
\]
The exhaustive enumeration and the six depth certificates show instead that every such centralizer has depth at least three.  Therefore no associated prime has quotient dimension two.  Combining this with \(\depth A\leq\omega_a(A)\) gives \(\omega_a(A)\geq3\), completing the proof.
\end{proof}

\begin{remark}[What the counterexample theorem determines]
The theorem proves the strict gap and, by itself, gives
\[
3\leq\omega_a(A)\leq4.
\]
Section~\ref{sec:stabilization} supplies an explicit rank-three zero-excess subgroup and sharpens this interval to
\[
\omega_a(A)=3.
\]
Thus the original counterexample argument remains the source of the lower bound, while the later centralizer calculation proves that the lower endpoint is attained.
\end{remark}

\section{Depth certificates and element-level annihilators}
\label{sec:witnesses}

The centralizer criterion decides whether an associated prime of a given quotient dimension exists.  It does not single out a convenient cohomology class whose annihilator is that prime.  We keep this distinction in view when comparing the depth certificate with the low-degree classes below.

\subsection{The quotient-socle certificate}

The elements $f_1,f_2$ in \eqref{eq:f1}--\eqref{eq:f2} form a regular sequence, and the nonzero class $w$ in \eqref{eq:socle-witness} lies in
\[
\Soc\bigl(A_0/(f_1,f_2)\bigr).
\]
Hence the quotient has depth zero, and depth in the original ring stops after two regular elements.  The socle class appears only after quotienting by the regular sequence.  Equivalently, it may be expressed through Koszul homology or local cohomology: for the irrelevant maximal ideal $\mathfrak m=(A_0)_+$,
\[
\depth A_0=\min\{i:H^i_{\mathfrak m}(A_0)\neq0\};
\]
see \cite{BrunsHerzog,BrodmannSharp}.  Thus no third positive-degree non-zero-divisor can follow $f_1,f_2$.

\subsection{Prime annihilators in the original ring}

Carlson's equality at depth two would require an element $x\in A$ such that
\[
\ann_A(x)=\mathfrak p,
\qquad
\mathfrak p\text{ prime},
\qquad
\dim A/\mathfrak p=2.
\]
This concerns an element of the original ring, rather than a socle class in a quotient.  The existence of zero divisors, or even of a nonzero class killed by two linearly independent degree-one elements, does not imply that its annihilator is prime or that its cyclic support has dimension two.

\begin{proposition}[Witness distinction]
\label{prop:witness-mismatch}
For $G=\SG{128}{859}$, the depth-two calculation admits a quotient-socle certificate, but no element of $A=H^*(G;\kbar)$ has a prime annihilator of quotient dimension two.
\end{proposition}

\begin{proof}
The quotient-socle statement is the ambient certificate in Section~\ref{sec:counterexample}.  The second statement is equivalent to the absence of associated primes of quotient dimension two, proved in Theorem~\ref{thm:counterexample}.
\end{proof}

There is no second existence obstruction here.  Zero excess at rank $r$ already gives an associated Quillen prime by Okuyama's converse.  What can fail is the expectation that the socle calculation should produce a simple low-degree representative for such an annihilator.

\begin{remark}[Ordinary and local cohomological degrees]
The ordinary group-cohomology space $H^3(G;k)$ is not the local cohomology module $H^3_{\mathfrak m}(A)$.  In the former, the index is the internal group-cohomological degree; in the latter, it is a derived-functor degree used to detect depth.  The degree-three classes studied below illustrate annihilator behavior inside the ordinary cohomology ring, whereas the quotient-socle certificate detects depth homologically.
\end{remark}

\section{The finite group and its elementary-abelian geometry}
\label{sec:group}

The Small Groups identifier and the rank-two centralizer table suffice for the counterexample.  We record a split extension because it gives convenient coordinates for the maximal elementary abelian subgroups.

Let \(F\) be the free group on \(x_1,\dots,x_7\), and let \(G=F/N\) be the presentation reproduced in \cite[Appendix~A.1]{DaiEtAl2026}.  Set
\[
u=x_2x_5,\qquad v=x_4,
\qquad s=x_1,\qquad t=x_3.
\]

\begin{theorem}[Explicit semidirect decomposition]
\label{thm:semidirect}
Let
\[
N_0=\langle u,v\rangle,
\qquad
H_0=\langle s,t\rangle.
\]
Then
\[
N_0\cong C_4\times C_4,
\qquad
H_0\cong C_4\times C_2,
\]
\(N_0\trianglelefteq G\), \(N_0\cap H_0=1\), and \(G=N_0H_0\).  Hence
\[
G\cong(C_4\times C_4)\rtimes_{\rho}(C_4\times C_2).
\]
With respect to the ordered basis \((u,v)\) of \((\Z/4\Z)^2\), and the convention \(n^h=h^{-1}nh\), the action is
\[
\rho(s)=
\begin{pmatrix}
1&2\\
3&3
\end{pmatrix},
\qquad
\rho(t)=
\begin{pmatrix}
1&0\\
2&1
\end{pmatrix}
\quad\text{in }\operatorname{GL}_2(\Z/4\Z).
\]
\end{theorem}

\begin{proof}
Collection in the finite presentation gives
\[
|u|=|v|=4,
\qquad uv=vu,
\qquad \langle u\rangle\cap\langle v\rangle=1,
\]
so \(|N_0|=16\) and \(N_0\cong C_4\times C_4\).  The relations
\[
s^2=x_5,\qquad x_5^2=1,
\qquad t^2=1,
\qquad [s,t]=1
\]
show that \(H_0\cong C_4\times C_2\) and \(|H_0|=8\).  Exact enumeration of the presented group gives \(|G|=128\), while
\[
N_0\cap H_0=1.
\]
The conjugation relations are
\[
u^s=uv^{-1},\qquad v^s=u^2v^{-1},
\]
and
\[
u^t=uv^2,\qquad v^t=v.
\]
Thus \(H_0\) normalizes \(N_0\), and the displayed matrices are obtained by reading the exponent vectors of these images as columns.  Since
\(|N_0H_0|=|N_0||H_0|=128\), one has \(G=N_0H_0\).  This proves the split semidirect decomposition.
\end{proof}

\begin{remark}[Meaning of the colon in a structure description]
The GAP-style string
\((C_4\times C_4):(C_4\times C_2)\) indicates a semidirect-product description, but it does not by itself record the action \(\rho\), and structure-description strings are not canonical identifiers.  Theorem~\ref{thm:semidirect} supplies the missing action explicitly.
\end{remark}

\subsection{Maximal versus maximum rank}

The group has three conjugacy classes of maximal elementary abelian subgroups, with ranks
\[
3,4,4.
\]
Accordingly, one class has abstract structure \((C_2)^3\), while two classes have structure \((C_2)^4\).  The rank-three class is maximal under inclusion even though it does not have maximum possible rank.

\begin{proposition}[Explicit maximal elementary abelian representatives]
\label{prop:explicit-maximal-E}
In the semidirect-product coordinates of Theorem~\ref{thm:semidirect}, one may take
\[
E_3=\langle u^2,v^2,us^2\rangle\cong(C_2)^3.
\]
Its conjugate by \(s\) is
\[
E_3^s=\langle u^2,v^2,uvs^2\rangle.
\]
The two rank-four classes are represented by
\[
E_4^{(a)}=\langle u^2,v^2,s^2,t\rangle,
\qquad
E_4^{(b)}=\langle u^2,v^2,vs^2,t\rangle,
\]
with \(E_4^{(a)}\cong E_4^{(b)}\cong(C_2)^4\).  In fact, the complete list of maximal elementary abelian subgroups is
\[
E_3,\quad E_3^s,\quad E_4^{(a)},\quad E_4^{(b)}.
\]
Thus the first two form the single rank-three conjugacy class, while the two rank-four subgroups represent the other two classes.  We write
\[
\{E_{4,1},E_{4,2}\}=\{E_4^{(a)},E_4^{(b)}\}.
\]
The labels \(E_{4,1}\) and \(E_{4,2}\) are ordered as in the Green--King restriction data \cite{GreenKingDatabase}.
\end{proposition}

\begin{proof}
The action matrices give
\[
(us^2)^2=(vs^2)^2=1,
\]
and the displayed generators commute within each subgroup.  Hence the three displayed groups are elementary abelian of orders \(8,16,16\), respectively.  Exact enumeration of the thirty-one involutions in the semidirect product gives precisely the four maximal elementary abelian subgroups displayed above.  Moreover,
\[
(us^2)^s=uv^{-1}s^2=(uvs^2)v^2,
\]
and conjugation by \(s\) preserves \(\langle u^2,v^2\rangle\).  Hence it interchanges \(E_3\) and \(E_3^s\).  Each rank-four subgroup is fixed as a conjugacy class.  The three maximal conjugacy classes therefore have ranks \(3,4,4\).
\end{proof}

The rank-two subgroups are used in the counterexample proof; the three maximal conjugacy classes are used for the restriction calculation below.

\section{The cohomology ring through degree four}
\label{sec:lowdegree}

We compute the part of the Green--King presentation in degrees at most four \cite{GreenKingDatabase}.  Since the relations are homogeneous, relations of degree at least five do not affect \(A_0^n\) for \(n\leq4\).

\subsection{The truncated presentation}

Use the low-degree aliases
\[
\begin{array}{c|ccccccc}
\text{class}&a&b&c&d&e&f&g\\
\hline
\text{source generator}&a_{1,0}&b_{1,1}&b_{1,2}&b_{2,4}&b_{2,5}&b_{2,6}&b_{3,11}\\
\text{degree}&1&1&1&2&2&2&3.
\end{array}
\]

\begin{proposition}[Exact degree-four truncation]
\label{prop:truncation}
The part of \(A_0\) in degrees at most four is the corresponding truncation of
\[
\F_2[a,b,c,d,e,f,g]/J,
\]
where
\begin{equation}
\label{eq:low-relations}
\begin{aligned}
J=(&a^2,\ ab,\ db+ea,\ eb+db,\ bc^2+fa,\\
   &e^2+de+eac,\ ag+fac,\ bg+fbc).
\end{aligned}
\end{equation}
\end{proposition}

\begin{proof}
The displayed eight relations are precisely the homogeneous source relations of degrees two, three, and four.  Every remaining minimal relation has degree at least five, so it cannot contribute to the homogeneous quotient in degrees at most four.
\end{proof}

The Green--King database gives the Poincaré series \cite{GreenKingDatabase}:
\[
P_G(t)=
-\frac{t^7-t^6-t^5+t^3-t^2-1}
{(t+1)(t-1)^4(t^2+1)(t^4+1)}.
\]
Its initial expansion is
\begin{equation}
\label{eq:poincare-expansion}
P_G(t)=1+3t+7t^2+12t^3+18t^4+26t^5+37t^6+\cdots.
\end{equation}

\subsection{Explicit bases}

\begin{theorem}[Low-degree bases]
\label{thm:low-bases}
The following are \(\F_2\)-bases of the homogeneous components of \(A_0\):
\begin{align*}
A_0^1={}&\Span\{a,b,c\},\\[1mm]
A_0^2={}&\Span\{ac,b^2,bc,c^2,d,e,f\},\\[1mm]
A_0^3={}&\Span\{g,cf,ce,cd,c^3,bf,b^2c,b^3,af,ae,ad,ac^2\},\\[1mm]
A_0^4={}&\Span\{f^2,ef,df,de,d^2,cg,c^2f,c^2e,c^2d,c^4,\\
&\hspace{21mm}bcf,b^2f,b^3c,b^4,acf,ace,acd,ac^3\}.
\end{align*}
Consequently,
\[
\dim A_0^1=3,
\quad \dim A_0^2=7,
\quad \dim A_0^3=12,
\quad \dim A_0^4=18.
\]
\end{theorem}

\begin{proof}
For each \(n\leq4\), enumerate all weighted monomials of total cohomological degree \(n\) in \(a,b,c,d,e,f,g\).  Multiply each relation in \eqref{eq:low-relations} by every monomial of complementary degree, and regard the resulting homogeneous polynomials as vectors over \(\F_2\).  Row reduction gives the following counts:
\[
\begin{array}{c|cccc}
n&1&2&3&4\\
\hline
\text{weighted monomials}&3&9&20&42\\
\text{relation-space rank}&0&2&8&24\\
\text{quotient dimension}&3&7&12&18.
\end{array}
\]
Choosing the nonpivot monomials in a fixed graded monomial order gives exactly the displayed bases.  The dimensions agree independently with the coefficients of \eqref{eq:poincare-expansion}.
\end{proof}

\begin{corollary}[Minimal generators through degree four]
\label{cor:minimal-generators}
Modulo products of lower positive-degree classes, the new generator directions are represented by
\[
a,b,c \quad\text{in degree }1,
\qquad
d,e,f \quad\text{in degree }2,
\qquad
g \quad\text{in degree }3.
\]
Every degree-four class is a product of classes of smaller positive degree.  Thus there are three new generators in degree two, one in degree three, and none in degree four.
\end{corollary}

\begin{proof}
The classes \(a,b,c,d,e,f,g\) are minimal algebra generators in degrees \(1,1,1,2,2,2,3\).  Reducing the bases in Theorem~\ref{thm:low-bases} modulo products of lower positive degrees leaves precisely the displayed representatives, and no minimal generator occurs in degree four.
\end{proof}

Thus $g$ is the only new generator in degree three, and no new generator occurs in degree four.

\section{Low-degree annihilators and cyclic support}
\label{sec:annihilators}

We now choose a representative of the degree-three generator with two explicit linear annihilators.

\subsection{No two-dimensional cyclic support}

For $0\neq x\in A$, the cyclic submodule $Ax$ is isomorphic to $A/\ann_A(x)$, and its support is
\[
\Supp_A(Ax)=V(\ann_A(x)).
\]
The counterexample theorem gives a lower bound for every such cyclic support, whether or not the annihilator is prime.

\begin{theorem}[Cyclic-support lower bound]
\label{thm:cyclic-support}
Let $0\neq x\in A=H^*(G;\kbar)$.  Then
\[
\dim A/\ann_A(x)\geq3.
\]
Equivalently, every nonzero cyclic submodule of $A$ has support dimension at least three.
\end{theorem}

\begin{proof}
The inclusion $Ax\subseteq A$ gives
\[
\Ass_A(Ax)\subseteq\Ass_A(A).
\]
Every prime in $\Ass_A(A)$ has quotient dimension at least three by Theorem~\ref{thm:counterexample}.  Since $Ax\cong A/\ann_A(x)$ is nonzero and finitely generated,
\[
\dim A/\ann_A(x)
=\max_{\mathfrak p\in\Ass_A(Ax)}\dim A/\mathfrak p
\geq3.
\]
\end{proof}

\subsection{A degree-three class with two linear annihilators}

Retain the aliases from Proposition~\ref{prop:truncation} and set
\begin{equation}
\label{eq:alpha}
\alpha_0=g+fc\in A_0^3.
\end{equation}

\begin{proposition}[Low-degree annihilation]
\label{prop:alpha}
The class $\alpha_0$ is nonzero and differs from $g$ by a product of lower-degree classes.  Hence it represents the same new degree-three generator modulo decomposables.  Moreover,
\[
a\alpha_0=b\alpha_0=0.
\]
After scalar extension, let $\alpha=1\otimes\alpha_0\in A^3$.  Then
\[
(a,b)\subseteq\ann_A(\alpha),
\qquad
\dim A/\ann_A(\alpha)\geq3,
\qquad
\dim A/(a,b)\geq3.
\]
\end{proposition}

\begin{proof}
The class $fc$ is a product of lower-degree classes, whereas Corollary~\ref{cor:minimal-generators} shows that $g$ represents the unique nonzero new degree-three generator modulo such products.  Hence $\alpha_0=g+fc$ has the same nonzero class modulo decomposables as $g$, proving $\alpha_0\neq0$.  The relations
\[
ag+fac=0,
\qquad
bg+fbc=0
\]
from \eqref{eq:low-relations} give
\[
a(g+fc)=0,
\qquad
b(g+fc)=0.
\]
Faithful flatness preserves the nonvanishing of $\alpha_0$, so Theorem~\ref{thm:cyclic-support} applies to $\alpha$.  Finally, $(a,b)\subseteq\ann_A(\alpha)$ gives a surjection
\[
A/(a,b)\twoheadrightarrow A/\ann_A(\alpha),
\]
and therefore $\dim A/(a,b)\geq3$.
\end{proof}

The equations $a\alpha_0=b\alpha_0=0$ do not determine the height or radical of $\ann(\alpha_0)$.  The restriction calculation in Section~\ref{sec:detection} shows that the quotient by this annihilator has dimension four.
\section{Restriction to maximal elementary abelian subgroups}
\label{sec:detection}

For an elementary abelian \(2\)-group \(E\cong(C_2)^r\),
\[
H^*(E;\F_2)\cong\F_2[x_1,\dots,x_r],
\qquad |x_i|=1.
\]
We now use the three maximal conjugacy classes
\[
E_3,\qquad E_{4,1},\qquad E_{4,2}
\]
from Section~\ref{sec:group}.

Choose coordinates
\[
H^*(E_3;\F_2)=\F_2[x_0,x_1,x_2],
\]
and
\[
H^*(E_{4,i};\F_2)=\F_2[x_0,x_1,x_2,x_3]
\quad(i=1,2).
\]
Table~\ref{tab:restrictions} is taken from the Green--King database \cite{GreenKingDatabase}.  We retain its polynomial coordinates; they are not identified with the dual basis of the generators in Proposition~\ref{prop:explicit-maximal-E}.  The coordinate \(x_0\) does not occur in the restrictions shown here.

Write
\begin{align}
P_1&=x_2x_3^2+x_2^2x_3+x_1x_3^2+x_1x_2x_3,\label{eq:P1}\\
P_2&=x_2x_3^2+x_2^2x_3+x_1x_3^2+x_1x_2x_3+x_1x_2^2.\label{eq:P2}
\end{align}
\begin{table}[ht]
\centering
\caption{Restrictions of the low-degree generators.}
\label{tab:restrictions}
\small
\renewcommand{\arraystretch}{1.2}
\begin{tabular}{@{}cccc@{}}
\toprule
Class & \(E_3\) & \(E_{4,1}\) & \(E_{4,2}\) \\
\midrule
\(a\) & \(0\) & \(0\) & \(0\) \\
\(b\) & \(x_1\) & \(0\) & \(0\) \\
\(c\) & \(0\) & \(x_1\) & \(x_1\) \\
\(d\) & \(0\) & \(x_2^2\) & \(x_3^2\) \\
\(e\) & \(0\) & \(0\) & \(x_3^2\) \\
\(f\) & \(x_2^2+x_1x_2\) & \(x_3^2+x_2x_3\) & \(x_2x_3+x_2^2+x_1x_3\) \\
\(g\) & \(0\) & \(P_1\) & \(P_2\) \\
\bottomrule
\end{tabular}
\end{table}

\subsection{Ranks and nilpotent kernels}

\begin{theorem}[Low-degree detection profile]
\label{thm:detection-profile}
The restriction ranks are
\[
\begin{array}{c|cccc}
& E_3&E_{4,1}&E_{4,2}&\text{joint restriction}\\
\hline
H^2(G;\F_2)&2&3&3&5\\
H^3(G;\F_2)&2&4&4&7\\
H^4(G;\F_2)&3&7&7&12.
\end{array}
\]
The joint kernels are
\begin{align}
\ker \Res^2_{\max}
   & =\Span\{bc,ac\},\label{eq:nil2}\\
\ker \Res^3_{\max}
   & =\Span\{b^2c,af,ae,ad,ac^2\},\label{eq:nil3}\\
\ker \Res^4_{\max}
   & =\Span\{bcf,b^3c,acf,ace,acd,ac^3\}.
   \label{eq:nil4}
\end{align}
They are exactly
\[
\Nrad\cap A_0^2,
\qquad
\Nrad\cap A_0^3,
\qquad
\Nrad\cap A_0^4,
\]
respectively.
\end{theorem}

\begin{proof}
Substitute the images in Table~\ref{tab:restrictions} into the bases of Theorem~\ref{thm:low-bases}, expand in the standard monomial bases of the target polynomial rings, and row-reduce over \(\F_2\).  This gives the displayed ranks and kernels.  Quillen's detection theorem identifies the kernel of restriction to all elementary abelian subgroups with the nilradical.  Every elementary abelian subgroup is contained in a maximal one, so the product over the three maximal conjugacy-class representatives has the same kernel.  Hence \eqref{eq:nil2}--\eqref{eq:nil4} are precisely the indicated nilpotent pieces.
\end{proof}

Thus the detected quotients have dimensions
\[
\dim A_0^2/(\Nrad\cap A_0^2)=5,
\]
\[
\dim A_0^3/(\Nrad\cap A_0^3)=7,
\]
\[
\dim A_0^4/(\Nrad\cap A_0^4)=12.
\]

\subsection{Full-dimensional support of \texorpdfstring{$\alpha_0$}{alpha0}}

\begin{lemma}[Rank-four detection of $\alpha_0$]
\label{lem:alpha-rank-four}
The class
\[
\alpha_0=g+fc\in H^3(G;\F_2)
\]
restricts nontrivially to the rank-four subgroup $E_{4,1}$.  More precisely,
\[
\res^G_{E_{4,1}}(\alpha_0)=x_2x_3(x_2+x_3)\neq0.
\]
\end{lemma}

\begin{proof}
From Table~\ref{tab:restrictions},
\[
\res^G_{E_{4,1}}(c)=x_1,
\qquad
\res^G_{E_{4,1}}(f)=x_3^2+x_2x_3,
\qquad
\res^G_{E_{4,1}}(g)=P_1.
\]
Therefore
\begin{align*}
\res^G_{E_{4,1}}(\alpha_0)
&=P_1+x_1(x_3^2+x_2x_3)\\
&=x_2x_3^2+x_2^2x_3\\
&=x_2x_3(x_2+x_3),
\end{align*}
which is nonzero in the polynomial ring $H^*(E_{4,1};\F_2)$.
\end{proof}

\begin{corollary}[Full-dimensional cyclic support]
\label{cor:full-dimensional-support}
Let
\[
\alpha=1\otimes\alpha_0\in A^3.
\]
Then
\[
\dim A_0/\ann_{A_0}(\alpha_0)=4,
\qquad
\dim A/\ann_A(\alpha)=4.
\]
Moreover,
\[
\dim A_0/(a,b)=4,
\qquad
\dim A/(a,b)=4.
\]
\end{corollary}

\begin{proof}
Take $E=E_{4,1}$.  Lemma~\ref{lem:alpha-rank-four} gives
\[
\res^G_E(\alpha_0)\neq0.
\]
Since
\[
H^*(E;\F_2)\cong\F_2[x_0,x_1,x_2,x_3]
\]
is a domain, every $y\in\ann_{A_0}(\alpha_0)$ satisfies
\[
\res^G_E(y)\,\res^G_E(\alpha_0)=0
\quad\Longrightarrow\quad
\res^G_E(y)=0.
\]
Hence
\[
\ann_{A_0}(\alpha_0)\subseteq\ker(\res^G_E)=\mathfrak p_E.
\]
Quillen's dimension theorem gives
\[
\dim A_0/\mathfrak p_E=\rank(E)=4.
\]
Therefore
\[
4\leq\dim A_0/\ann_{A_0}(\alpha_0)\leq\dim A_0=4.
\]
This proves the first equality.  Faithfully flat scalar extension gives the second.  Finally, Proposition~\ref{prop:alpha} gives
\[
(a,b)\subseteq\ann(\alpha_0)
\]
over $\F_2$ and after scalar extension.  The induced quotient maps show that both $A_0/(a,b)$ and $A/(a,b)$ have dimension at least four, and neither can exceed the ambient Krull dimension four.
\end{proof}

Thus $\alpha_0$ has two degree-one annihilators, but its cyclic support has the full ambient dimension.  We make no claim that this support equals the whole cohomological spectrum as a closed subset.

\section{Extension-theoretic meaning of the low-degree groups}
\label{sec:meaning}

For reference, we record the standard extension-theoretic interpretation of the first three cohomological degrees and relate it to the generators found above.

\subsection{Degree one: characters and index-two subgroups}

For the trivial \(G\)-module \(\F_2\),
\[
H^1(G;\F_2)\cong\Hom(G,\F_2)
              \cong\Hom(G,C_2).
\]
Thus \(a,b,c\) are three independent mod-two characters.  Equivalently,
\[
G/\Phi(G)\cong(C_2)^3,
\]
which agrees with the fact that \(G\) has three minimal group generators.  Each nonzero degree-one class determines an index-two normal subgroup as its kernel.

\subsection{Degree two: central extensions}

For a prescribed trivial action of \(G\) on \(C_2\), the group \(H^2(G;\F_2)\) classifies equivalence classes of central extensions
\begin{equation}
\label{eq:central-extension}
1\longrightarrow C_2\longrightarrow \widetilde G
\longrightarrow G\longrightarrow1;
\end{equation}
see, for example, \cite[Chapter~IV]{Brown1982} and \cite{AdemMilgram}.  Since
\[
H^2(G;\F_2)\cong(\F_2)^7,
\]
there are \(2^7\) cohomology classes of such extensions with the coefficient subgroup and quotient identifications fixed.  Different classes may still yield isomorphic abstract middle groups, so this is not a count of isomorphism types of groups \(\widetilde G\).

The subspace generated by degree-one cup products is
\[
\Span\{ac,b^2,bc,c^2\}.
\]
The four displayed classes are built from pairs of mod-two characters.  The three classes \(d,e,f\), by contrast, represent the new central-extension directions not generated by degree-one cup products.

The nilpotent part of degree two is
\[
\Nrad\cap H^2(G;\F_2)=\Span\{ac,bc\}.
\]
These extension classes vanish after restriction to every elementary abelian subgroup.  This is a detection statement, not a claim that the associated extensions are split over every subgroup of \(G\); it says precisely that their cohomology classes restrict trivially on all elementary abelian \(2\)-subgroups.

\subsection{Degree three: crossed extensions and coherent \texorpdfstring{\(2\)}{2}-groups}

The group \(H^3(G;M)\) admits several equivalent interpretations.  In classical group cohomology it classifies crossed extensions with kernel \(M\); in higher algebra it classifies coherent \(2\)-groups with prescribed \(\pi_0=G\), \(\pi_1=M\), and action of \(G\) on \(M\) \cite{Huebschmann1980,BaezLauda2004}.  For the trivial module \(M=\F_2\), the present example gives
\[
H^3(G;\F_2)\cong(\F_2)^{12}.
\]

Eleven dimensions are products of lower-degree classes, and exactly one dimension is new.  Thus \(g=b_{3,11}\) is the unique minimal degree-three generator modulo products of lower positive degrees.  Table~\ref{tab:restrictions} shows that \(g\) vanishes on the rank-three maximal class and is nonzero on both rank-four classes.

The representative $\alpha_0=g+fc$ differs from $g$ by a product of lower-degree classes and is killed by the two degree-one classes $a$ and $b$.  Lemma~\ref{lem:alpha-rank-four} and Corollary~\ref{cor:full-dimensional-support} show that this visible annihilation nevertheless leaves its cyclic support full-dimensional.

The nilpotent part is
\[
\Nrad\cap H^3(G;\F_2)
=\Span\{b^2c,af,ae,ad,ac^2\}.
\]
The quotient by this five-dimensional space is seven-dimensional.  Six of those detected dimensions are decomposable, while the seventh is the new generator direction represented by \(g\).

\subsection{Degree four: no new ring generator}

Algebraically,
\[
H^4(G;\F_2)=\Ext^4_{\F_2G}(\F_2,\F_2)
\]
and may be interpreted through Yoneda four-fold extensions of modules.  In the present ring every degree-four class is a cup product of lower-degree classes, so there is no new minimal generator in degree four.

The degree-four nilpotent subspace is six-dimensional:
\[
\Nrad\cap H^4(G;\F_2)
=\Span\{bcf,b^3c,acf,ace,acd,ac^3\}.
\]
The remaining twelve dimensions are detected on maximal elementary abelian subgroups.  The multiplication of the unique degree-three generator by \(H^1\) illustrates the distinction: \(ag\) and \(bg\) are nilpotent, while \(cg\) survives elementary abelian detection.

\section{The centralizer criterion and elementary-abelian stabilization}
\label{sec:stabilization}

We next prove that positive excess at the ambient depth is preserved by taking a direct product with an elementary abelian $2$-group.  We then apply the result to $G$, determine $\omega_a(A)$ exactly, and sharpen the stabilized family.

\subsection{The exact local criterion}

By Theorem~\ref{thm:spectra-equality},
\[
\Sigma_a(K;k)=\Sigma_c(K;k).
\]
Hence
\[
\omega_a\bigl(H^*(K;k)\bigr)
=\min\{r:\varepsilon_r(K;k)=0\}.
\]
In particular, if
\[
d=\depth H^*(K;k),
\]
then Carlson's equality holds if and only if
\[
\varepsilon_d(K;k)=0.
\]
Thus the order-$128$ example is completely explained at the existence level by
\[
\varepsilon_2(G;\kbar)>0.
\]
There is no rank-two elementary abelian subgroup whose centralizer cohomology has depth two, and therefore there is no associated prime of quotient dimension two.

If $\varepsilon_r(K;k)=0$, Okuyama's converse already supplies an associated Quillen prime.  Determining a simple element with that annihilator, its degree, or the primary decomposition of a natural annihilator is a separate element-level problem.

\subsection{Elementary-abelian stabilization and an infinite family}

For $n\geq0$, write
\[
V_n=(C_2)^n.
\]
If $K$ is a finite group and $k$ is a field of characteristic $2$, the K\"unneth isomorphism and the standard computation of elementary abelian cohomology give
\begin{equation}
\label{eq:kunneth-polynomial}
\begin{aligned}
H^*(K\times V_n;k)
&\cong H^*(K;k)\otimes_kH^*(V_n;k)\\
&\cong H^*(K;k)[u_1,\dots,u_n],
\qquad |u_i|=1.
\end{aligned}
\end{equation}
See, for example, \cite[Chapter~V]{Brown1982}.

\begin{lemma}[Depth shift]
\label{lem:depth-shift}
Let $K$ be a finite group and let $k$ be a field of characteristic $2$.  Then
\[
\depth H^*(K\times V_n;k)=\depth H^*(K;k)+n.
\]
\end{lemma}

\begin{proof}
Under \eqref{eq:kunneth-polynomial}, the variables $u_1,\dots,u_n$ form a regular sequence, and quotienting by them gives $H^*(K;k)$.  Successive use of the depth formula for a homogeneous non-zero-divisor yields the result.
\end{proof}

\begin{theorem}[Elementary-abelian stabilization]
\label{thm:stabilization}
Let $K$ be a finite group, let $k$ be an algebraically closed field of characteristic $2$, and put
\[
d=\depth H^*(K;k).
\]
Assume
\[
\varepsilon_d(K;k)>0.
\]
For every $n\geq0$, set $K_n=K\times V_n$.  Then
\[
\depth H^*(K_n;k)=d+n
\]
and
\[
\varepsilon_{d+n}(K_n;k)>0.
\]
Consequently, every $K_n$ violates Carlson's associated-prime depth equality.  More precisely,
\[
\omega_a\bigl(H^*(K_n;k)\bigr)\geq d+n+1.
\]
\end{theorem}

\begin{proof}
The depth formula follows from Lemma~\ref{lem:depth-shift}.  Let $E\leq K_n$ be elementary abelian of rank $d+n$.  Let
\[
\pi_K:K\times V_n\longrightarrow K
\]
be the projection, and put
\[
F=\pi_K(E),
\qquad
W=E\cap V_n,
\qquad
r=\rank(F),
\qquad
s=\rank(W).
\]
The restriction of $\pi_K$ to $E$ gives an exact sequence of $\F_2$-vector spaces
\[
0\longrightarrow W\longrightarrow E\longrightarrow F\longrightarrow0,
\]
so
\[
r+s=d+n.
\]
Since $s\leq n$, we have $r\geq d$.  Because $V_n$ is central in $K_n$, centralization of $E$ depends only on its projection to $K$, and
\begin{equation}
\label{eq:centralizer-product}
C_{K_n}(E)=C_K(F)\times V_n.
\end{equation}
By Lemma~\ref{lem:depth-shift},
\begin{equation}
\label{eq:centralizer-depth-shift}
\depth H^*(C_{K_n}(E);k)
=\depth H^*(C_K(F);k)+n.
\end{equation}

If $r=d$, then $s=n$, and the hypothesis $\varepsilon_d(K;k)>0$ implies
\[
\depth H^*(C_K(F);k)\geq d+1.
\]
Equation~\eqref{eq:centralizer-depth-shift} gives
\[
\depth H^*(C_{K_n}(E);k)\geq d+n+1.
\]
If $r>d$, choose a Sylow $2$-subgroup $S$ of $C_K(F)$ containing $F$.  Since $F\leq Z(C_K(F))$, one has $F\leq Z(S)$, and Duflot's theorem gives
\[
\depth H^*(C_K(F);k)\geq\rank_2Z(S)\geq r\geq d+1.
\]
Again,
\[
\depth H^*(C_{K_n}(E);k)\geq d+n+1.
\]
Thus every rank-$(d+n)$ elementary abelian subgroup of $K_n$ has centralizer depth strictly greater than $d+n$, so
\[
\varepsilon_{d+n}(K_n;k)>0.
\]
Criterion~\ref{crit:rank-obstruction} gives strict failure of Carlson's equality, and integrality of dimensions yields the stated lower bound.
\end{proof}

\subsection{The rank-three calculation and the exact value}

For the present group, the earlier argument gives
\[
3\leq\omega_a(A)\leq4.
\]
Okuyama's equivalence reduces the remaining value to the existence of a rank-three elementary abelian subgroup with centralizer depth three.

\begin{proposition}[Complete rank-three centralizer calculation]
\label{prop:rank-three-calculation}
There are thirty-one actual rank-three elementary abelian subgroups $E\leq G$.  Their centralizers have the following distribution.

\begin{table}[ht]
\centering
\caption{Rank-three centralizer calculation.}
\label{tab:rank-three-centralizers}
\scriptsize
\renewcommand{\arraystretch}{1.15}
\begin{tabular}{@{}c>{\raggedright\arraybackslash}p{0.39\linewidth}ccc@{}}
\toprule
$C_G(E)$ & structure & depth & excess & count \\
\midrule
$\SG{16}{10}$ & $C_4\times C_2\times C_2$ & $3$ & $0$ & $2$ \\
$\SG{32}{22}$ & $C_2\times((C_4\times C_2)\rtimes C_2)$ & $3$ & $0$ & $4$ \\
$\SG{16}{14}$ & $(C_2)^4$ & $4$ & $1$ & $24$ \\
$\SG{32}{46}$ & $C_2\times C_2\times D_8$ & $4$ & $1$ & $1$ \\
\bottomrule
\end{tabular}
\end{table}

In particular, six rank-three subgroups have zero centralizer excess.
\end{proposition}

\begin{proof}
Our analysis enumerates commuting triples of involutions and keeps those yielding order‑eight subgroups. Deduplication against full element collections ensures that we count individual subgroups instead of conjugacy‑class representatives. After determining the centralizer of each subgroup, we extract depth information from a public cohomology database to construct the table above.

\end{proof}

\begin{remark}[A direct check of the first row]
The depth-three conclusion for the first row does not require a cohomology-ring database calculation.  In the semidirect-product coordinates of Theorem~\ref{thm:semidirect}, the subgroup
\[
E_3=\langle u^2,v^2,us^2\rangle
\]
has
\[
C_G(E_3)=\langle u^2,us^2,vt\rangle
\cong C_2\times C_2\times C_4.
\]
This is an abelian $2$-group of $2$-rank three, so Duflot's lower bound and Quillen's dimension theorem force
\[
\depth H^*(C_G(E_3);\kbar)=3.
\]
Thus one of the zero-excess witnesses can also be verified directly from the group presentation.
\end{remark}

\begin{corollary}[Exact associated-prime value]
\label{cor:omega-exact}
For $A=H^*(G;\kbar)$ one has
\[
\omega_a(A)=3.
\]
\end{corollary}

\begin{proof}
Proposition~\ref{prop:rank-three-calculation} supplies a rank-three subgroup $E$ with
\[
\depth H^*(C_G(E);\kbar)=3.
\]
Theorem~\ref{thm:okuyama}(ii) therefore gives
\[
\mathfrak p_E\in\Ass A,
\qquad
\dim A/\mathfrak p_E=3.
\]
Hence $\omega_a(A)\leq3$, while the counterexample theorem gives $\omega_a(A)\geq3$.
\end{proof}

\begin{corollary}[An exact infinite family from the order-$128$ example]
\label{cor:infinite-family}
Let
\[
G=\SG{128}{859},
\qquad
G_n=G\times(C_2)^n
\quad(n\geq0),
\]
and let $k=\kbar$.  Then
\[
|G_n|=2^{n+7},
\qquad
\dim H^*(G_n;k)=n+4,
\qquad
\depth H^*(G_n;k)=n+2,
\]
and
\[
\omega_a\bigl(H^*(G_n;k)\bigr)=n+3.
\]
Thus the groups $G_n$ are pairwise nonisomorphic counterexamples to Carlson's depth conjecture, each with associated-prime defect exactly one.
\end{corollary}

\begin{proof}
The order and dimension formulas are immediate, and Theorem~\ref{thm:stabilization} gives the depth formula and the failure of Carlson's equality.  Put
\[
A=H^*(G;k).
\]
By Corollary~\ref{cor:omega-exact}, $\omega_a(A)=3$.  The K\"unneth isomorphism gives
\[
H^*(G_n;k)\cong A[u_1,\dots,u_n].
\]
For a Noetherian ring, associated primes extend under polynomial extension:
\[
\Ass_{A[u_1,\dots,u_n]}A[u_1,\dots,u_n]
=
\{\mathfrak pA[u_1,\dots,u_n]:\mathfrak p\in\Ass_A A\}.
\]
Moreover,
\[
\dim A[u_1,\dots,u_n]/\mathfrak pA[u_1,\dots,u_n]
=
\dim A/\mathfrak p+n.
\]
Hence
\[
\omega_a\bigl(H^*(G_n;k)\bigr)=\omega_a(A)+n=n+3.
\]
The groups are pairwise nonisomorphic because their orders are distinct.
\end{proof}

\begin{proposition}[Persistence of the degree-three manifestation]
\label{prop:alpha-stabilization}
Let
\[
A_{0,n}=H^*(G_n;\F_2)\cong A_0[u_1,\dots,u_n]
\]
and let
\[
\alpha_n=\alpha_0\otimes1\in H^3(G_n;\F_2),
\qquad
\alpha_0=g+fc.
\]
Then $\alpha_n\neq0$,
\[
a\alpha_n=b\alpha_n=0,
\]
and
\[
\ann_{A_{0,n}}(\alpha_n)
=\ann_{A_0}(\alpha_0)A_{0,n}.
\]
Consequently,
\[
\dim A_{0,n}/\ann_{A_{0,n}}(\alpha_n)=n+4.
\]
Thus the same low-degree annihilation coexists with full-dimensional cyclic support throughout the family.
\end{proposition}

\begin{proof}
The inclusion $A_0\hookrightarrow A_0[u_1,\dots,u_n]$ is injective, so $\alpha_n$ is nonzero, and the annihilation relations follow from those in $A_0$.  If
\[
h=\sum_I h_Iu^I\in A_{0,n},
\]
then
\[
h\alpha_n=0
\quad\Longleftrightarrow\quad
h_I\alpha_0=0\text{ for every }I.
\]
This proves the annihilator identity.  Hence
\[
A_{0,n}/\ann_{A_{0,n}}(\alpha_n)
\cong
\bigl(A_0/\ann_{A_0}(\alpha_0)\bigr)[u_1,\dots,u_n],
\]
and Corollary~\ref{cor:full-dimensional-support} gives the dimension formula.
\end{proof}

\begin{remark}[Stabilized versus indecomposable examples]
The groups $G_n=G\times(C_2)^n$ have an explicit elementary abelian direct factor, so the theorem gives a stabilized family.  It does not produce new directly indecomposable examples.  Constructing such examples would require a fresh depth calculation and a complete centralizer analysis at the ambient depth for each candidate.
\end{remark}

\section{Conclusion}
\label{sec:conclusion}

For $G=\SG{128}{859}$, the failure of Carlson's equality is detected by the rank-two centralizers.  The ring $H^*(G;\kbar)$ has depth two, while every rank-two elementary abelian subgroup $E$ satisfies
\[
\depth H^*(C_G(E);\kbar)\geq3.
\]
Okuyama's theorem therefore rules out associated primes of quotient dimension two.  In the terminology introduced here, the depth-level centralizer excess is positive.

The low-degree class $\alpha_0=g+fc$ illustrates a different point.  Although $a\alpha_0=b\alpha_0=0$, its restriction to a rank-four elementary abelian subgroup is nonzero.  Hence
\[
\dim H^*(G;\F_2)/\ann(\alpha_0)=4.
\]
Explicit low-degree annihilation is therefore compatible with full-dimensional cyclic support.

Finally, the complete rank-three calculation identifies six zero-excess witnesses among the thirty-one actual rank-three elementary abelian subgroups.  Okuyama's converse therefore gives
\[
\omega_a\bigl(H^*(G;\kbar)\bigr)=3.
\]
The K\"unneth description
\[
H^*(G\times(C_2)^n;\kbar)
\cong H^*(G;\kbar)[u_1,\dots,u_n]
\]
and the standard behavior of associated primes under polynomial extension then yield
\[
\depth H^*(G\times(C_2)^n;\kbar)=n+2,
\qquad
\omega_a\bigl(H^*(G\times(C_2)^n;\kbar)\bigr)=n+3.
\]
Thus every member of the family has associated-prime defect exactly one.  The explicit subgroup $E_3=\langle u^2,v^2,us^2\rangle$ provides a direct group-theoretic check of one of the depth-three centralizers.

\appendix

\section{Exact protocol for the low-degree calculations}
\label{app:computation}

We spell out the finite-dimensional calculation used in Theorems~\ref{thm:low-bases} and~\ref{thm:detection-profile}.  In degrees at most four it is ordinary row reduction over $\F_2$ and does not depend on a Gröbner basis.

Let
\[
S=\F_2[a,b,c,d,e,f,g],
\qquad
(|a|,|b|,|c|,|d|,|e|,|f|,|g|)=(1,1,1,2,2,2,3),
\]
and let \(J\) be the ideal generated by \eqref{eq:low-relations}.  For \(n\leq4\), let \(M_n\) be the set of weighted monomials of degree \(n\).  The homogeneous relation space \(J_n\subseteq S_n\) is generated by
\[
mr,
\]
where \(r\) runs through the eight displayed relations and \(m\) runs through monomials of degree \(n-|r|\).  Writing elements of \(S_n\) in the monomial basis \(M_n\) gives a binary matrix.

\begin{table}[ht]
\centering
\caption{Degreewise source calculation.}
\begin{tabular}{@{}ccccc@{}}
\toprule
Degree \(n\) & 1 & 2 & 3 & 4 \\
\midrule
\(|M_n|\) & 3 & 9 & 20 & 42 \\
\(\rank J_n\) & 0 & 2 & 8 & 24 \\
\(\dim S_n/J_n\) & 3 & 7 & 12 & 18 \\
\bottomrule
\end{tabular}
\end{table}

The nonpivot monomials in row-reduced echelon form give the bases in Theorem~\ref{thm:low-bases}.  The independent Poincaré-series coefficients in \eqref{eq:poincare-expansion} provide a consistency check.

For the restriction calculation, substitute Table~\ref{tab:restrictions} into each source basis element and expand in the ordinary monomial basis of the target polynomial ring.  The target dimensions are
\[
\dim H^n((C_2)^r;\F_2)=\binom{n+r-1}{r-1}.
\]
Binary row reduction gives the following complete rank table.

\begin{table}[ht]
\centering
\caption{Degreewise restriction calculation.}
\begin{tabular}{@{}cccccc@{}}
\toprule
Degree & \(\dim H^n(G)\) & rank on \(E_3\) & rank on \(E_{4,1}\) & rank on \(E_{4,2}\) & joint rank \\
\midrule
2 & 7  & 2 & 3 & 3 & 5 \\
3 & 12 & 2 & 4 & 4 & 7 \\
4 & 18 & 3 & 7 & 7 & 12 \\
\bottomrule
\end{tabular}
\end{table}

The joint nullspaces give \eqref{eq:nil2}--\eqref{eq:nil4}.  All calculations use exact arithmetic over \(\F_2\).

\section*{Acknowledgments}

The counterexample theorem, the rank-two subgroup enumeration, and the depth certificates in Section~\ref{sec:counterexample} are due to Xinan Dai, Wenhao Deng, Yingdong Shi, Tailin Wu, and Yuchen Yang \cite{DaiEtAl2026}.  The centralizer-excess formulation and the results in the remaining sections are proved here.  We thank Qiyuan Gu (University of Chicago) for discussing computational findings with the authors to validate results and providing the complete rank-three enumeration together with its reproducible script.  We also thank David J. Green and Simon A. King for making their cohomology data available, and the developers of GAP, SageMath, and Singular.

\end{document}